\documentclass{daj}
\usepackage{amssymb,amsmath,amsthm,amsxtra,calc,bm, color}
\usepackage{enumerate}
\usepackage[margin=.95in]{geometry}

\usepackage[scr=boondoxo]{mathalfa}
\usepackage{mathtools}

\usepackage{mleftright}
\mleftright

\usepackage{nccmath}
\usepackage{cases}

\usepackage{calc}

\def\Z{\mathbb{Z}}
\def\Q{\mathbb{Q}}
\def\R{\mathbb{R}}
\def\H{\mathbb{H}}

\def\C{\mathbb{C}}

\DeclareMathOperator{\im}{Im}
\DeclareMathOperator{\re}{Re}
\DeclareMathOperator{\Tr}{Tr}
\DeclareMathOperator{\Ai}{Ai}
\DeclareMathOperator{\arccosh}{arccosh}

\def\SL{{\rm SL}}
\def\PSL{{\rm PSL}}

\newcommand{\pfrac}[2]{\left(\frac{#1}{#2}\right)}
\newcommand{\pmfrac}[2]{\left(\mfrac{#1}{#2}\right)}

\renewcommand{\pmatrix}[4]{\left(\begin{smallmatrix}#1 & #2 \\ #3 & #4\end{smallmatrix}\right)}
\renewcommand{\bar}[1]{\overline{#1}}

\newcommand{\sumprime}{\sideset{}{'}\sum}

\newcommand{\fS}{\mathfrak{S}}

\DeclareMathOperator{\sgn}{sgn}

\def\ep{\varepsilon}

\newtheorem{theorem}{Theorem}[section]
\newtheorem{lemma}[theorem]{Lemma}
\newtheorem{corollary}[theorem]{Corollary}
\newtheorem{proposition}[theorem]{Proposition}

\theoremstyle{remark}
\newtheorem*{remark}{Remark}

\numberwithin{equation}{section}

\mathtoolsset{showonlyrefs}

\dajAUTHORdetails{%
  title = {Asymptotic Distribution of Traces of Singular Moduli}, 
  author = {Nickolas Andersen, William Duke},
  plaintextauthor = {Nickolas Andersen and William Duke},
  keywords = {singular moduli, Kloosterman sums},
}   

\dajEDITORdetails{%
   year={2022},
   number={4},
   received={5 November 2020},   
   published={3 June 2022},  
   doi={10.19086/da.33153},       
}   

\begin{document}

\begin{frontmatter}[classification=text]

\title{Asymptotic Distribution of Traces of Singular Moduli} 

\author[nick]{Nickolas Andersen\thanks{Supported by NSF grant DMS-1701638}}
\author[bill]{William Duke\thanks{Supported by NSF grant DMS-1701638}}

\begin{abstract}
	We determine the asymptotic behavior of twisted traces of singular moduli with a power-saving error term in both the discriminant and the order of the pole at $i\infty$.
	Using this asymptotic formula, we obtain an exact formula for these traces involving the class number and a finite sum involving the exponential function evaluated at CM points.
\end{abstract}
\end{frontmatter}

\maketitle

\section{Introduction}

The modular $j$-function
\begin{align}
	j(z) &= q^{-1} \prod\limits_{n=1}^\infty (1-q^n)^{-24}\Big(1+240\sum\limits_{n=1}^\infty \sum\limits_{m\mid n}m^3 q^n\Big)^3 \\
		&=  q^{-1} + 744 + 196884q + \ldots = \sum_{n=-1}^\infty c(n)q^n, \qquad q=e(z)=e^{2\pi i z},
\end{align}
is of fundamental importance in number theory.
These Fourier coefficients $c(n)$ are integral linear combinations of the dimensions of the irreducible representations of the monster group\footnote{For a comprehensive survey of this and related results, see \cite{moonshine} and the many references therein.} and its values at imaginary quadratic irrationalities are algebraic integers with well-known properties in class field theory.

In the 1930s, Petersson \cite{petersson} and Rademacher \cite{rademacher-j} independently discovered the formula
\begin{equation} \label{eq:c-n-formula}
	c(n) = 2\pi n^{-\frac 12} \sum_{c=1}^\infty \frac{K(n,c)}c I_{1}\pfrac{4\pi\sqrt n}{c},
\end{equation}
where $I_\nu$ is the $I$-Bessel function and $K(n,c)$ is the ordinary Kloosterman sum
\begin{equation}
	K(n,c) = \sum_{\substack{d\bmod c \\ (c,d)=1}} e\pfrac{\bar d + nd}{c}, \qquad d\bar d\equiv 1\pmod{c}.
\end{equation}
In his paper, Rademacher remarked that without \eqref{eq:c-n-formula}, the $c(n)$ ``can be found\ldots by troublesome computations, which for higher $n$ are practically inexecutable.''\footnote{He also commented that the coefficients ``do not seem to have attracted much attention before.'' Certainly the state of the subject has changed somewhat in the intervening years.}
The convergence of \eqref{eq:c-n-formula} is quite slow.
However, (in principle, at least) the formula can be used to compute $c(n)$ using a finite number of terms of the series since one knows a priori that $c(n)\in \Z$. 
The Weil bound $|K(n,c)|\leq \sigma_0(c)\sqrt c$, where $\sigma_0$ is the sum of divisors function, is enough to show that $C\sqrt n$ terms suffice as long as $n$ is large enough compared to $C$.

Also of great interest are certain special values of $j$.
For each negative fundamental discriminant $d$, let $z_d$ be the point in the complex upper half-plane $\H$ given by
\begin{equation}
	z_d = 
	\begin{cases}
		\tfrac 12 \sqrt{d} & \text{ if $d\equiv 0\pmod{4}$}, \\
		-\tfrac 12 + \tfrac 12 \sqrt{d} & \text{ if $d\equiv 1\pmod{4}$}.
	\end{cases}
\end{equation}
From the classical theory of complex multiplication we know that the singular moduli $j_1(z_d)$, where $j_1=j-744$, 
are algebraic integers of degree $h(d)$, the class number of $\Q(\sqrt{d})$.
In particular, the algebraic trace $\Tr j_1(z_d)$, i.e.~the sum of the Galois conjugates of $j_1(z_d)$, is an integer.

Kaneko \cite{kaneko-1,kaneko-2} expressed the coefficients $c(n)$ in terms of the traces $\Tr j_1(z_d)$ using a result of Zagier \cite{zagier-traces}, who showed that the traces appear as coefficients of weakly holomorphic modular forms of half-integral weight.
Zagier's results sparked significant interest in traces of singular moduli which has led to numerous papers on the subject (of which we will only mention a few).
Bruinier, Jenkins, and Ono \cite{bjo} proved a formula like \eqref{eq:c-n-formula} for the traces involving half-integral weight Kloosterman sums, and computations led them to conjecture the limit
\begin{equation} \label{eq:duke-24-original} 
	\lim_{d\to-\infty} \frac{1}{h(d)} \left(\Tr j_1(z_d) - \sum_{z_Q\in \mathcal R(1)} e(-z_Q)\right) = -24,
\end{equation}
where $\mathcal R(Y)$ is the rectangle
\begin{equation}
	\mathcal R(Y) = \left\{z=x+iy\in \H : -\tfrac 12 \leq x < \tfrac 12 \text{ and } y>Y\right\}.
\end{equation}
Their conjecture was confirmed by the second author in \cite{duke-24}.
The limit \eqref{eq:duke-24-original} converges very slowly.
We will see in Corollary~\ref{cor} below that by modifying and lengthening the exponential sum we can obtain a more quickly convergent limit; in particular for any $\ep>0$, the trace $\Tr j_1(z_d)$ is the nearest integer to the quantity
\begin{equation} \label{eq:nearest-int}
	-24h(d) + \sum_{z_Q\in \mathcal R(|d|^{-1+\epsilon})} \left( e(-z_Q) - e(-\bar z_Q) \right)
\end{equation}
as long as $|d|$ is sufficiently large with respect to $\ep$.
For example, when $d=-303$, we have $h(d) = 10$ and
\begin{equation}
	\Tr j_1(z_{-303}) = -561\,766\,949\,784\,377\,042\,888\,940,
\end{equation}
while \eqref{eq:nearest-int} with $\ep=\frac 1{100}$ equals
\begin{equation}
	-561\,766\,949\,784\,377\,042\,888\,939.643\ldots.
\end{equation}

In our main theorems we treat traces of any weakly holmorphic modular form with integer coefficients and give estimates that are uniform in the order of its pole at $i\infty$.
We also allow twists by genus characters.
The main series from which these results originate (\eqref{eq:trace-sinh} below) resembles \eqref{eq:c-n-formula}, except that it involves half-integral weight Kloosterman sums. 
However, for our results it is not enough to use the Weil bound for Kloosterman sums to prove either \eqref{eq:duke-24-original} or our improved results.
The proof of \eqref{eq:duke-24-original} in \cite{duke-24} uses the uniform distribution of CM points \cite{duke-half-integral}, but to obtain stronger estimates it is essential to measure the cancellation among the Kloosterman sums directly.

\section{Statement of Results}

Our results improve and generalize the limit formula \eqref{eq:duke-24-original}.
We begin by fixing notation and explaining the more general setting.
Each conjugate of $j_1(z_d)$ is of the form $j_1(z_Q)$, where $Q=[a,b,c]$ is a positive definite integral binary quadratic form of discriminant $d=b^2-4ac$, and
\begin{equation}
	z_Q = \frac{-b+\sqrt{d}}{2a}.
\end{equation}
In fact, we can choose $z_Q\in \mathcal F$, where
\begin{equation}
	\mathcal F = \left\{ z\in \H : -\tfrac 12 \leq \re(z) \leq 0 \text{ and } |z|\geq 1 \right\} \cup \left\{ z\in \H : 0 < \re(z) < \tfrac 12 \text{ and } |z| > 1 \right\}
\end{equation}
is the usual fundamental domain for the action of $\PSL_2(\Z)$ on $\H$.
This point of view leads to a straightforward generalization to non-fundamental discriminants $d\equiv 0,1\pmod{4}$ by defining
\begin{equation}
	\Tr j_1(z_d) = \sumprime_{z_Q\in \mathcal F} j_1(z_Q),
\end{equation}
where $Q$ runs over all positive definite integral binary quadratic forms of discriminant $d$ with $z_Q\in \mathcal F$.
The primed sum indicates that terms are weighted by $1/\omega_Q$, where $\omega_Q=1$ unless $Q$ is $\PSL_2(\Z)$-equivalent to $[a,0,a]$ or $[a,a,a]$, in which case it equals $2$ or $3$, respectively.

The main result of \cite{duke-24} determines the asymptotic behavior of $\Tr f(z_d)$ for any $f\in \C[j]$.
A convenient basis for $\C[j]$ is given by the functions $j_m = P_m(j)$, where $P_m(x)$ are the Faber polynomials defined by the property that $j_m = q^{-m}+O(q)$.
The first few basis elements are $j_0 = 1$, $j_1 = j-744$, and
\begin{align}
	j_2 &= j^2 - 1488j + 159768, \\
	j_3 &= j^3 - 2232 j^2 + 1069956 j - 36866976.
\end{align}
We generalize the results of \cite{duke-24} by considering sums twisted by genus characters.
For any factorization $D=dd'$ of the negative discriminant $D$, where $d$ is a (positive or negative) fundamental discriminant and $d'$ is a discriminant, there is an associated character
\begin{equation}
	\chi_d(Q) = 
	\begin{cases}
		\pfrac dn & \text{ if $(a,b,c,d)=1$ and $Q$ represents $n$ and $(d,n)=1$}, \\
		0 & \text{ if $(a,b,c,d)>1$}.
	\end{cases}
\end{equation}
We say that $Q$ represents $n$ if $n=ax^2+bxy+cy^2$ for some $x,y\in \Z$.
We define the twisted traces of $j_m$ by
\begin{equation}
	\Tr_d j_m(z_D) = \sumprime_{z_Q\in \mathcal F} \chi_d(Q) j_m(z_Q).
\end{equation}
Let $\delta_1=1$ and $\delta_d=0$ otherwise, and let $\sigma_1(m)$ denote the sum of the divisors of $m$.
Finally, let $\theta\in [0,\frac 7{64}]$ be an admissible exponent toward the Ramanujan conjecture for Maass cusp forms of integral weight.

\begin{theorem} \label{thm:twisted-m-1}
	For each negative discriminant $D$, let $d$ be any fundamental discriminant dividing $D$. Then for each $m\geq 1$ we have
	\begin{equation} \label{eq:twisted-m-1}
		\Tr_d j_m(z_D) - \sum_{z_Q\in \mathcal R(\frac 1m)} \chi_d(Q) e(-mz_Q) = -24\delta_d\sigma_1(m)h(D) + O\left(|D|^{\frac {17}{36}+\ep}m^{\frac 56+\frac\theta3+\ep}\right).
	\end{equation}
\end{theorem}

In Theorem~\ref{thm:twisted-m-1} we have subtracted the quantity $\sum \chi_d(Q)e(-mz_Q)$ to provide an easier comparison with \eqref{eq:duke-24-original},
but our methods are optimized for the following modification.

\begin{theorem} \label{thm:trace-R(Y)}
	Let $D$, $d$, and $m$ be as in Theorem~\ref{thm:twisted-m-1}.
	Then for $0<Y\ll \frac 1m$ we have
	\begin{multline}
		\Tr_d j_m(z_D) - \sum_{z_Q\in \mathcal R(Y)} \chi_d(Q) \left(e(-mz_Q)-e(-m\bar z_Q)\right) 
		\\
		= -24\delta_d\sigma_1(m)h(D) + O\left(m|D|^{\frac 13}Y^{\frac 12} \left( Y^{-\frac 16} + |D|^{\frac 5{36}}m^{\frac 13(1+\theta)} \right)(m|D|/Y)^\ep\right).
	\end{multline}
\end{theorem}

\begin{remark}
	If we assume the Lindel\"of hypothesis for the central values $L(\tfrac 12,f\times \chi)$ and $L(\tfrac 12,\chi)$, where $f$ is a weight $0$ Maass cusp form and $\chi$ is a quadratic Dirichlet character, we can replace the exponent $\frac 5{36}$ in Theorem~\ref{thm:trace-R(Y)} by $\frac 1{12}$.
	The convexity bound would yield an exponent of $\frac 16$ instead.
	If we assume the Linnik-Selberg conjecture for half-integral weight Kloosterman sums (namely, that the sum in Theorem~\ref{thm:Sk-sums} is bounded above by $(|mn|x)^\ep$) we can replace the error term in Theorem~\ref{thm:trace-R(Y)} by $O(m|D|^{1/4}Y^{1/2})$.
\end{remark}

Choosing $Y$ small enough that the error term tends to zero, we obtain the following.

\begin{corollary} \label{cor}
	Let $D$, $d$ and $m$ be as in Theorem~\ref{thm:twisted-m-1} and let $Y = C m^{-A}|D|^{-B}$, where $A>3$, $B>1$, and $C>0$ are constants.
	Then $\Tr_d j_m(z_D)$ is the nearest integer to
	\begin{equation}
		-24 \delta_d \sigma_1(m) h(D) + \sum_{z_Q \in \mathcal R(Y)} \chi_d(Q) \left(e(-mz_Q) - e(-m\bar z_Q)\right)
	\end{equation}
	provided $m|D|$ is sufficiently large compared to $C$.
\end{corollary}

We prove Theorem~\ref{thm:trace-R(Y)} by obtaining a hybrid estimate for sums of half-integral weight Kloosterman sums associated with the Kohnen plus space on $\Gamma_0(4)$.
Much of the heavy lifting that leads to this estimate was done in the authors' recent work \cite{andersen-duke-kloosterman} studying real quadratic analogues of traces of singular moduli.
In that paper, we obtained asymptotic formulas for averages of two types of real quadratic geometric invariants: integrals of $j_m$ over modular surfaces and contour integrals of $j_m$ along the boundaries of the surfaces.
Essential here and in \cite{andersen-duke-kloosterman} is a variant of Kuznetsov's formula relating Kloosterman sums to products $\rho_j(d)\rho_j(d')$ of coefficients of Maass forms in Kohnen's plus space, where $D=dd'$.
The main significant difference is that here $d$ and $d'$ have opposite sign, while in the real quadratic case $d$ and $d'$ have the same sign.

\section{From traces to quadratic Weyl sums} \label{sec:weyl}

We begin by relating the twisted traces of singular moduli to the quadratic Weyl sums
\begin{equation} \label{eq:Tm-def}
	T_m(d,d';c) = \sum_{\substack{b\bmod c \\ b^2\equiv D(c)}} \chi_d\left(\left[\tfrac c4,b,\tfrac{b^2-D}{c}\right]\right) e\pmfrac{2mb}{c},
\end{equation}
where, as in the introduction, $D=dd'$.
In \cite[(4.10)--(4.11)]{DIT-cycle} the function $j_m(z)$ is expressed in terms of (the analytic continuation of) a Poincar\'e series $G_{-m}(z,s)$ evaluated at $s=1$:
\begin{equation}
	j_m(z) = G_{-m}(z,1) - 24\sigma_1(m).
\end{equation}
Thus by Proposition~4 of \cite{DIT-cycle} we have\footnote{In \cite{DIT-cycle} the Weyl sums are denoted $S_m(d,d';c)$ but we have used the notation $T_m(d,d';c)$ to avoid confusion with the Kloosterman sums in the next section.}
\begin{equation} \label{eq:limit-Tm}
	\Tr_d j_m(z_D) = -24\delta_{d}\sigma_1(m)h(D) + \pi(2m)^{\frac12} |D|^{\frac14} \lim_{s\to 1^+} \sum_{4\mid c}  \mfrac{T_m(d,d';c)}{\sqrt c} I_{s-\frac12} \left(\mfrac{4\pi m}c |D|^{\frac 12}\right),
\end{equation}
where $I_\nu$ is the $I$-Bessel function.
In Section~\ref{sec:kuznetsov} we will prove the following estimate for averages of the Weyl sums $T_m(d,d';c)$.

\begin{theorem} \label{thm:Tm-sums}
Suppose that $D=dd'$ is a negative discriminant and that $d$ is a fundamental discriminant.
Then for any $m\geq 1$ we have
\begin{equation}
	\sum_{4\mid c\leq x} \frac{T_m(d,d';c)}{\sqrt c} \ll \left(x^{\frac 16} + |D|^{\frac 29}m^{\frac 13(1+\theta)}\right)(m|D|x)^\ep.
\end{equation}
\end{theorem}

Theorem~\ref{thm:Tm-sums} and \cite[(10.29.1) and (10.30.1)]{nist} together justify exchanging the sum and the limit in \eqref{eq:limit-Tm}, so we conclude that
\begin{equation} \label{eq:trace-sinh}
	\Tr_d j_m(z_D) = -24\delta_{d}\sigma_1(m)h(D) + \sum_{4\mid c} T_m(d,d';c) \sinh\left(\mfrac{4\pi m}c |D|^{\frac 12}\right).
\end{equation}
This formula generalizes \cite[(12)]{duke-24}.
Suppose that $x \gg m|D|^{1/2}$.
For the tail of the $c$-sum in \eqref{eq:trace-sinh} we have, by partial summation,
\begin{align} \label{eq:Tm-tail}
	\sum_{4\mid c\geq x} T_m(d,d';c) \sinh\left(\mfrac{4\pi m}c |D|^{\frac 12}\right) \ll m|D|^{\frac 12} x^{-\frac 12}\left( x^{\frac 16} + |D|^{\frac 29} m^{\frac 13(1+\theta)} \right)(m|D|x)^\ep.
\end{align}
Setting $x= \frac{2}{Y}|D|^{1/2}$, Theorem~\ref{thm:trace-R(Y)} will follow after we relate the terms $c< x$ in \eqref{eq:trace-sinh} to the CM points $z_Q\in \mathcal R(Y)$.

\begin{lemma} \label{lem:R(Y)}
 For any $Y>0$ and $\xi\in \{-1,1\}$ we have
\begin{equation} \label{eq:c-leq-x-zq-sum}
	\mfrac 12\sum_{4\mid c < \frac 2Y|D|^{1/2}} T_m(d,d';c) \exp\left(\xi\mfrac{4\pi m}{c}|D|^{\frac 12}\right) = \sum_{z_Q\in \mathcal R(Y)} \chi_d(Q) e(-mz_{Q,\xi}),
\end{equation}
where $z_{Q,\xi}=z_Q$ when $\xi=1$ and $z_{Q,\xi}=\bar z_Q$ when $\xi=-1$.
\end{lemma}

\begin{proof}
Let $x=\frac2Y|D|^{1/2}$.
	Replacing $c$ by $4c$ and using the definition \eqref{eq:Tm-def} we have
	\begin{equation} \label{eq:Tm-partial-sum-1}
		\mfrac 12\sum_{c < \frac x4 } T_m(d,d';4c) \exp\left(\xi\mfrac{\pi m}{c}|D|^{\frac 12}\right) 
		= \sum_{c < \frac x4} \sum_{\substack{b\bmod 2c \\ b^2\equiv D(4c)}} \chi_d \left(\left[c,b,\mfrac{b^2-D}{4c}\right]\right) e\left(\mfrac{-m(-b+i\xi\sqrt{|D|})}{2c}\right),
	\end{equation}
	where we have used that the $b$-summands are unchanged by replacing $b$ by $b+2c$.
	The quadratic forms $Q$ of discriminant $D$ with $-\frac 12 \leq \re(z_Q) < \frac 12$ are in one-to-one correspondence with pairs of integers $(b,c)$ for which $-c < b \leq c$ and $\frac{b^2-D}{4c}\in \Z$.
	Furthermore, $\im(z_Q) = \frac 1{2c}|D|^{1/2}$, so the condition $c < \frac x4$ is equivalent to $\im(z_Q) > \frac{2}{x}|D|^{1/2}$.
	Thus the left-hand side of \eqref{eq:c-leq-x-zq-sum} equals
	\begin{equation}
		\sum_{z_Q\in \mathcal R(\frac 2x |D|^{1/2})} \chi_d(Q) e(-mz_{Q,\xi}),
	\end{equation}
	from which the lemma follows after replacing $x$ by $\frac 2{Y}|D|^{1/2}$.
\end{proof}

\begin{proof}[Proof of Theorem~\ref{thm:trace-R(Y)}]
	By \eqref{eq:trace-sinh}, Lemma~\ref{lem:R(Y)}, and \eqref{eq:Tm-tail} with $x = \frac{2}{Y}|D|^{1/2}$, we find that
	\begin{multline}
		\Tr_d j_m(z_D)
		= -24\delta_d\sigma_1(m)h(D)
		+ \sum_{z_Q\in \mathcal R(Y)} \chi_d(Q) \left(e(-mz_Q)-e(-m\bar{z}_Q)\right) \\ + O\left(m|D|^{\frac 13}Y^{\frac 12} \left( Y^{-\frac 16} + |D|^{\frac 5{36}}m^{\frac 13(1+\theta)} \right)(m|D|/Y)^\ep\right),
	\end{multline}
	as desired.
\end{proof}

\begin{proof}[Proof of Theorem~\ref{thm:twisted-m-1}]
	Setting $Y=1/m$ in Theorem~\ref{thm:trace-R(Y)}, we obtain \eqref{eq:twisted-m-1} with $e(-mz_Q)$ replaced by $e(-mz_Q)-e(-m\bar z_Q)$.
	By Lemma~\ref{lem:R(Y)} it remains to estimate the sum
	\begin{equation} \label{eq:Tm-neg-sum}
		\sum_{c\leq 2m|D|^{1/2}} T_m(d,d';c)\exp\left(-\mfrac{4\pi m}{c}|D|^{\frac 12}\right).
	\end{equation}
	A straightforward argument involving Theorem~\ref{thm:Tm-sums} and partial summation shows that the sum in \eqref{eq:Tm-neg-sum} is $\ll |D|^{\frac {17}{36}+\ep}m^{\frac 56+\frac\theta3+\ep}$.
\end{proof}

\section{Plus-space Kuznetsov trace formula with opposite signs} \label{sec:kuznetsov}

The estimate in Theorem~\ref{thm:Tm-sums} will follow from an estimate for sums of Kloosterman sums of half-integral weight associated to Kohnen's plus space on $\Gamma_0(4)$.
For additional background see Section~3 of \cite{andersen-duke-kloosterman}.
Let $k=\pm \frac 12 = \lambda+\frac 12$ and suppose that $(-1)^\lambda m, (-1)^\lambda n \equiv 0,1\pmod{4}$.
The plus-space Kloosterman sums are
\begin{equation}
	S_k^+(m,n;c) = e\left(-\mfrac k4\right) \sum_{d\bmod c} \pfrac cd \ep_d^{2k} e\pfrac{m\bar d+nd}{c} \times 
	\begin{cases}
		1 & \text{ if }8\mid c, \\
		2 & \text{ if }4\mid\mid c,
	\end{cases}
\end{equation}
where $d\bar d\equiv 1\pmod{c}$ and $\ep_d=1$ or $i$ according to $d\equiv 1$ or $3 \pmod 4$, respectively.
These are related to the quadratic Weyl sums via Kohnen's identity (Lemma~8 of \cite{kohnen-newforms})
\begin{equation} \label{eq:kohnen-identity}
	T_m(d,d';c) = \sum_{n\mid (m,\frac c4)} \pmfrac dc \sqrt{\mfrac{2n}{c}} \, S_{\frac 12}^+ \left(d',\mfrac{m^2}{n^2}d;\mfrac{c}{n}\right).
\end{equation}
Furthermore, we have the relation
\begin{equation} \label{eq:Sk-minus}
	S_k^+(m,n;c) = S_k^+(n,m;c) = S_{-k}^+(-m,-n;c).
\end{equation}
Theorem~\ref{thm:Tm-sums} follows in a straightforward way from \eqref{eq:kohnen-identity}, \eqref{eq:Sk-minus}, and the following theorem.

\begin{theorem}\label{thm:Sk-sums}
Let $k,\lambda,m,n$ be as above, with the additional assumption that $m>0$, $n<0$, and $(-1)^\lambda m=dv^2$, $(-1)^\lambda n=d'w^2$, with $d,d'$ fundamental discriminants.
Then
\begin{equation}
	\sum_{4\mid c\leq x} \frac{S_k^+(m,n;c)}{c} \ll \left( x^{\frac 16} + |dd'|^{\frac 29}(vw)^{\frac 13(1+\theta)} \right)(|mn|x)^\ep.
\end{equation}
\end{theorem}

Individually, the Kloosterman sums satisfy the Weil bound
\begin{equation} \label{eq:weil-bound}
	|S_k^+(m,n,c)| \leq 2 \sigma_0(c) \gcd(m,n,c)^{\frac 12} \sqrt{c},
\end{equation}
see Lemma~6.1 of \cite{dfi-weyl}.
Thus the sum in Theorem~\ref{thm:Sk-sums} is trivially bounded above by $x^{1/2}|mnx|^\ep$.

To prove Theorem~\ref{thm:Sk-sums} we will use a version of Kuznetsov's formula relating the Kloosterman sums $S_k^+(m,n;c)$ to Fourier coefficients of Maass cusp forms residing in the Kohnen plus-space.
Let $\vartheta(z) = \sum_{n\in \Z}e(n^2z)$ denote the usual Jacobi theta function, and let $\nu_\vartheta$ denote the associated multiplier system.
Let $\Gamma=\Gamma_0(4)$ and $(k,\nu) = (\frac 12, \nu_{\vartheta})$ or $(-\frac 12, \bar\nu_{\vartheta})$.
The plus-space $\mathcal V_k^+$ of Maass cusp forms of weight $k$ for $\Gamma$ is spanned by functions $u:\H\to\C$ satisfying
\begin{itemize}
	\item $(\Delta_k  + \frac 14+r^2)u = 0$, where $\Delta_k = y^2(\partial_x^2 + \partial_y^2) - iky \partial_x$ is the hyperbolic Laplacian,
	\item $u(\gamma z) = \nu(\gamma)j(\gamma,z)^k u(z)$ for all $\gamma\in \Gamma$, where $j(\pmatrix abcd,z)=\frac{cz+d}{|cz+d|}$,
	\item the Fourier coefficients of $u(z)$ are supported on exponents $\equiv 0,(-1)^\lambda\pmod{4}$,
	\item $\lVert u \rVert <\infty$, where $\lVert \cdot \rVert$ is the Petersson norm.
\end{itemize}
The quantity $r$ is called the spectral parameter.

Once and for all we fix a spectrally normalized ($\lVert u \rVert = 1$) orthonormal basis $\{u_j\}_{j=0}^\infty$ of $\mathcal V_k^+$ consisting of eigenforms for the Hecke operators, ordered by eigenvalue $\frac 14+r_j^2$.
The spectral parameter $r_0=\frac i4$ corresponds to $u_0=y^{1/4}\vartheta(z)$ or its conjugate.
In either case, the product of the $m$-th and $n$-th Fourier coefficients of $u_0$ vanishes whenever $mn<0$.
Thus $u_0$ does not contribute to the Kuznetsov formula.
For $j\geq 1$, Theorem~1.2 of \cite{baruch-mao} shows that there is a unique normalized Maass cusp form $v_j$ of weight~$0$ with spectral parameter $2r_j$ which is even if $k=\frac 12$ and odd if $k=-\frac 12$, and such that the Hecke eigenvalues of $u_j$ and $v_j$ agree.
Since each spectral parameter in weight~$0$ on $\SL_2(\Z)$ is positive we find that $r_1 > 0$.
For $j\geq 1$ we normalize the Fourier coefficients $\rho_j$ of $u_j$ by
\begin{equation}
	u_j(z) = \sum_{n\neq 0} \rho_j(n) W_{\frac k2\sgn(n),ir_j}(4\pi |n|y) e(nx),
\end{equation}
where $W_{\mu,\nu}$ is the $W$-Whittaker function.

Kuznetsov proved his formula in \cite{kuznetsov} in the setting of integral weight on $\SL_2(\Z)$.
Later Proskurin \cite{proskurin-new} generalized Kuznetsov's formula to arbitrary weight and for Fuchsian groups of the first kind, but only when both inputs $m,n$ are positive.
Blomer \cite[Proposition~2]{blomer} proved a version of the Kuznetsov formula for $\nu\in \{\nu_\vartheta,\bar\nu_\vartheta\}$ with $mn<0$. 
Using Blomer's result, and following the exact same procedure\footnote{The most difficult part of the argument in \cite{andersen-duke-kloosterman} is the proof of Proposition~5.6, but that result already holds for all $m,n\equiv 0,1\pmod{4}$ with no sign restrictions.} as in Section~5 of \cite{andersen-duke-kloosterman}, we obtain the plus-space version for opposite signs, Theorem~\ref{thm:KTFplus-space+} below.
We first fix some notation.
Given a smooth test function $\varphi:[0,\infty)\to \R$ satisfying
\begin{gather} \label{eq:varphi-cond}
	\varphi(0) = \varphi'(0) = 0 \quad \text{ and } \quad
	\varphi^{(j)}(x) \ll x^{-2-\varepsilon} \quad \text{ for }j=0,1,2,3,
\end{gather}
define the integral transform
\begin{align}
	\check \varphi(r) &= 2\cosh(\pi r) \int_0^\infty K_{2ir}(x) \varphi(x) \mfrac{dx}{x},
\end{align}
where $K_{2ir}$ is the $K$-Bessel function.
If $d$ is a fundamental discriminant, let $\chi_d = \pfrac d\cdot$ and let $L(s,\chi_d)$ denote the analytic continuation of the Dirichlet $L$-function
\begin{equation}
 	L(s,\chi_d) = \sum_{n=1}^\infty \frac{\chi_d(n)}{n^s}.
\end{equation} 
Finally, we define
\begin{equation}
	\fS_d(w,s) = \sum_{\ell\mid w} \mu(\ell)  \frac{\chi_{d}(\ell)}{\sqrt \ell} \sum_{ab\ell=w} \pfrac ab^s.
\end{equation}

\begin{theorem} \label{thm:KTFplus-space+}
Let $\varphi:[0,\infty)\to\R$ be a smooth test function satisfying \eqref{eq:varphi-cond}.
Let $k=\pm \frac 12=\lambda+\frac 12$.
Suppose that $m>0$, $n<0$ with $(-1)^\lambda m,(-1)^\lambda n\equiv 0,1\pmod{4}$, and write
\begin{equation}
	(-1)^\lambda m = v^2 d', \quad (-1)^\lambda n = w^2 d, \quad \text{ with }d,d'\text{ fundamental discriminants.}
\end{equation}
Fix an orthonormal basis of Maass cusp forms $\{u_j\}\subset \mathcal V_k^+$ with associated spectral parameters $r_j$ and coefficients $\rho_j(n)$.
Then
\begin{multline}
	\sum_{4\mid c>0} \frac{S_k^+(m,n,c)}{c} \varphi \pfrac{4\pi\sqrt{|mn|}}{c}
	=
	6\sqrt{|mn|} \sum_{j\geq 1} \frac{\overline{\rho_{j}(m)}\rho_{j}(n)}{\cosh\pi r_j} \check\varphi(r_j)
	\\
	+ \mfrac 12\int_{-\infty}^\infty \left|\frac{d}{d'}\right|^{ir} \frac {L(\frac 12-2ir,\chi_{d'})L(\frac 12+2ir,\chi_{d})\fS_{d'}(v,2ir)\fS_{d}(w,2ir)}{|\zeta(1+4ir)|^2 \cosh \pi r |\Gamma(\frac{1-k}{2}+ir)\Gamma(\frac{1+k}2-ir)|} \check\varphi(r) \, dr.
\end{multline}
\end{theorem}

\section{Proof of Theorem~\ref{thm:Sk-sums}} \label{sec:proof}

Let $a=4\pi\sqrt{|mn|}$ and $x\geq3$ and let $1\leq T \leq x/3$ be a free parameter to be chosen later.
Following \cite{sarnak-tsimerman}, we fix a test function $\varphi=\varphi_{a,x,T}:[0,\infty)\to[0,1]$ satisfying
\begin{enumerate}[(i)] \setlength\itemsep{.5em}
	\item $\varphi(t)=1$ for $\mfrac{a}{2x}\leq t\leq \mfrac ax$,
	\item $\varphi(t)=0$ for $t\leq \mfrac{a}{2x+2T}$ and $t\geq \mfrac{a}{x-T}$,
	\item $\varphi'(t) \ll \left( \mfrac{a}{x-T} - \mfrac ax \right)^{-1} \ll \mfrac{x^2}{aT}$, and
	\item $\varphi$ and $\varphi'$ are piecewise monotonic on a fixed number of intervals (whose number is independent of $a,x,T$).
\end{enumerate}
We apply the plus space Kuznetsov formula in Theorem~\ref{thm:KTFplus-space+} with this test function and we estimate each of the terms on the right-hand side.
For this, we require an estimate for the integral transform $\check\varphi(r)$.
All but the first estimate in the following theorem are proved in \cite[Section~6]{aa-kloosterman}.
There are some minor errors in that proof, and we have provided the corrections, along with the proof of the first estimate, in Appendix~\ref{sec:appendix}.
\begin{theorem} \label{thm:phicheck}
Suppose that $a,x,T$, and $\varphi=\varphi_{a,x,T}$ are as above.
Then
\begin{equation} \label{eq:phi-est}
	\check\varphi(r) \ll
	\begin{dcases}
		r^{-\frac32} & \text{ if } r\leq 1, \\
		e^{-\frac 12r} & \text{ if } \, 1\leq r\leq \mfrac{a}{8x}, \\
		r^{-1} & \text{ if } \, \max\big(\mfrac a{8x},1\big) \leq r\leq \mfrac ax, \\
		r^{-\frac 32}\min\Big( 1, \mfrac x{rT} \Big) & \text{ if } \, r\geq \max\big(\mfrac ax, 1\big).
	\end{dcases}	
\end{equation}
\end{theorem}

We first give two estimates for the contribution from the Maass cusp forms
\begin{equation}
	\mathcal K^m = \sqrt{|mn|} \sum_{j\geq 1} \frac{\overline{\rho_{j}(m)}\rho_{j}(n)}{\cosh\pi r_j} \check\varphi(r_j).
\end{equation}
The first estimate is Theorem~\ref{thm:maass-young-est} below, which we quote directly from \cite[Section~6]{andersen-duke-kloosterman}. 
This estimate uses Young's \cite{young-subconvexity} Weyl-type hybrid subconvexity estimate for central values of $L$-functions of Maass cusp forms of integral weight for $\SL_2(\Z)$ twisted by Dirichlet characters (see also Appendix~A of \cite{andersen-duke-kloosterman}).

\begin{theorem} \label{thm:maass-young-est}
Let $k=\pm \frac 12=\lambda+\frac 12$.
Suppose that $m>0$, $n<0$ and write $(-1)^\lambda m=v^2 d'$ and $(-1)^\lambda n=w^2 d$ with $d,d'$ fundamental discriminants.
Then
\begin{equation}\label{eq:maass-young-est}
	\sqrt{|mn|} \sum_{r_j \leq X} \frac{|\rho_j(m)\rho_j(n)|}{\cosh\pi r_j} \ll |dd'|^{\frac 16} (vw)^{\theta} X^2 (|mn|X)^\varepsilon.
\end{equation}
\end{theorem}

Our second estimate comes from Section~4 of \cite{andersen-duke-kloosterman}.

\begin{theorem} \label{thm:mve-general}
Suppose that $(k,\nu)=(\frac 12,\nu_\vartheta)$ or $(-\frac 12, \bar\nu_\vartheta)$.
Then for all $n\neq 0$ we have
\begin{equation} \label{eq:mve}
	n \sum_{X\leq r_j\leq 2X} |\rho_{j}(n)|^2 e^{-\pi r_j} \ll
		X^{-k\sgn(n)}\left(X^2 + |n|^{\frac 12+\varepsilon}\right)X^{\varepsilon}.
\end{equation}
\end{theorem}

The estimation of the Maass cusp form contribution follows the general structure of the proof of Proposition~5 of \cite{sarnak-tsimerman} and the proof of Theorem~9.1 of \cite{aa-kloosterman}.
We split the sum $\mathcal K^m$ into three ranges
\begin{equation}
	\mathcal K^m_1 = \sum_{1\leq r_j \ll \frac{a}{x}}, \quad \mathcal K^m_2 = \sum_{r_j\asymp \frac ax}, \quad \mathcal K^m_3 = \sum_{r_j \gg \frac ax},
\end{equation}
corresponding to the second, third, and fourth ranges in Theorem~\ref{thm:phicheck}.
Recall that $a\asymp \sqrt{|mn|}$.
For the first range we use Theorem~\ref{thm:maass-young-est} to obtain
\begin{align}
	\mathcal K^m_1 &\ll \sqrt{|mn|} \sum_{r_j\geq 1} \frac{|\rho_{j}(m)\rho_{j}(n)|}{\cosh\pi r_j} e^{-r_j/2} \\
	&\ll \sqrt{|mn|} \sum_{T=1}^\infty e^{-T/2} \sum_{T\leq r_j \leq T+1} \frac{|\rho_{j}(m)\rho_{j}(n)|}{\cosh\pi r_j} \ll |dd'|^{\frac 16}(vw)^\theta|mn|^\ep.
\end{align}
In the second range, again using Theorem~\ref{thm:maass-young-est} we have
\begin{equation}
	\mathcal K^m_2 \ll \sqrt{|mn|} \sum_{r_j\asymp \frac ax} \frac{|\rho_{j}(m)\rho_{j}(n)|}{\cosh\pi r_j} r_j^{-1} \ll x \sum_{r_j\asymp \frac ax} \frac{|\rho_{j}(m)\rho_{j}(n)|}{\cosh\pi r_j} \ll |dd'|^{\frac 23}(vw)^{1+\theta}x^{-1}(|mn|x)^\ep.
\end{equation}

To estimate $\mathcal K^m_3$ we consider the dyadic sums
\begin{equation}
	\mathcal K^m(A) = \sqrt{|mn|} \sum_{A\leq r_j < 2A} \frac{\bar{\rho_j(m)}\rho_j(n)}{\cosh\pi r_j} \check\varphi(r_j)
\end{equation}
for $A\geq 1$.
Applying Cauchy-Schwarz and Theorem~\ref{thm:mve-general} we obtain the estimate:
\begin{equation}
  	\sqrt{|mn|} \sum_{r_j \leq A} \frac{|\rho_j(m)\rho_j(n)|}{\cosh\pi r_j} \ll \left(A^2+|mn|^{\frac 14}A\right)(|mn|A)^\varepsilon.
\end{equation} 
Together with Theorem~\ref{thm:maass-young-est}, this implies that
\begin{align}
	\sqrt{|mn|} \sum_{A\leq r_j < 2A} \frac{|\rho_j(m)\rho_j(n)|}{\cosh\pi r_j} 
	& \ll \min\left(A^2+|mn|^{\frac 14}A, |dd'|^{\frac 16}(vw)^\theta A^2\right)(|mn|A)^\ep
	\\ &\ll \left(A^2 + |dd'|^{\frac {5}{24}}(vw)^{\frac 14+\frac 12\theta}A^{\frac 32}\right)(|mn|A)^\varepsilon,
\end{align}
where we have used that $\min(y,z)\leq (yz)^{1/2}$.
By Theorem~\ref{thm:phicheck} it follows that
\begin{equation}
	\mathcal K^m(A) \ll \min\left(1,\frac{x}{AT}\right)\left(A^{\frac 12} + |dd'|^{\frac {5}{24}}(vw)^{\frac 14+\frac 12\theta}\right)(|mn|A)^\varepsilon,
\end{equation}
so by summing the dyadic pieces we obtain
\begin{equation}
	\mathcal K^m_3 \ll \left(x^{\frac 12}T^{-\frac 12} + |dd'|^{\frac 5{24}}(vw)^{\frac 14+\frac 12\theta}\right) (|mn|x)^{\ep}.
\end{equation}
In total,
\begin{equation} \label{eq:K-m-est}
	\mathcal K^m \ll \left(|dd'|^{\frac 23}(vw)^{1+\theta}x^{-1} + x^{\frac 12}T^{-\frac 12} + |dd'|^{\frac 5{24}}(vw)^{\frac 14+\frac 12\theta}\right) (|mn|x)^{\ep}.
\end{equation}

We turn to the estimate of the integral
\begin{equation}
	\mathcal K^e = \int_{-\infty}^\infty \left|\frac{d}{d'}\right|^{ir} \frac {L_{d'}(-r)L_d(r)\fS_{d'}(v,2ir)\fS_{d}(w,2ir)}{|\zeta(1+4ir)|^2 \cosh \pi r |\Gamma(\frac{1-k}{2}+ir)\Gamma(\frac{1+k}2-ir)|} \check\varphi(r) \, dr,
\end{equation}
where, for brevity, we have written $L(\frac 12+2ir,\chi_{d})=L_d(r)$.
By symmetry it suffices to estimate the integrals $\mathcal K^e_0 = \int_0^1$ and $\mathcal K^e_1 = \int_1^\infty$.
Estimating the divisor sums trivially we find that $|\fS_d(w,s)| \leq \sigma_0(w)^2 \ll w^\ep$.
For $r\in (0,1]$ we have 
\[
|\zeta(1+4ir)|^2\gg r^{-2} \text{ and } \cosh\pi r |\Gamma(\tfrac{1-k}{2}+ir)\Gamma(\tfrac{1+k}2-ir)| \gg 1,
\]
so by Theorem~\ref{thm:phicheck} we have the estimate
\begin{equation} \label{eq:Ke0-est}
	\mathcal K^e_0 \ll (vw)^\varepsilon \int_0^1 \left|L_{d'}(r) L_d(r)\right| \, dr.
\end{equation}
Since $\cosh\pi r |\Gamma(\frac{1-k}{2}+ir)\Gamma(\frac{1+k}2-ir)| \sim \pi$ and $|\zeta(1+4ir)|^{-1}\ll r^{\varepsilon}$ for large $r$ we have by Theorem~\ref{thm:phicheck} that
\begin{equation}
	\mathcal K^e_1 \ll (vw)^\varepsilon \int_1^\infty \big|L_{d'}(r)L_d(r)\big| \mfrac{dr}{r^{3/2-\varepsilon}} + (vw)^\ep \int_{a/(8x)}^{a/x}|L_{d'}(r)L_{d}(r)| \mfrac{dr}{r^{1-\ep}}.
\end{equation}
In the first integral we multiply each Dirichlet $L$-function by $r^{-3/8}$ and the last factor by $r^{3/4}$. 
Applying H\"older's inequality in the case $\frac 16+\frac 16+\frac 23=1$ to both integrals, we obtain
\begin{multline}\label{eq:Ke1-holder}
	\mathcal K^e_1 \ll (vw)^\varepsilon \left(\int_1^\infty |L_{d'}(r)|^6 \mfrac{dr}{r^{9/4}}\right)^{\frac 16} \left(\int_1^\infty |L_{d}(r)|^6 \mfrac{dr}{r^{9/4}}\right)^{\frac 16}  \left(\int_1^\infty \mfrac{dr}{r^{9/8-\varepsilon}}\right)^{\frac 23} \\
	+ (vw)^\ep \sum_{\frac{a}{8x} \leq T < \frac{a}{x}} \left(\int_T^{T+1} |L_{d'}(r)|^6 \, dr\right)^{\frac 16}\left(\int_T^{T+1} |L_{d}(r)|^6 \, dr\right)^{\frac 16} \left(\int_T^{T+1} \frac{dr}{r^{3/2-\ep}}\right)^{\frac 23}.
\end{multline}
Young \cite{young-subconvexity} proved that
\begin{equation}
	\int_{T}^{T+1} |L_d(r)|^6 \, dr \ll (|d|(1+T))^{1+\varepsilon},
\end{equation}
from which it follows that
\begin{equation}
	\int_1^\infty |L_d(r)|^6 \mfrac{dr}{r^{9/4}} \leq \sum_{T=1}^\infty \mfrac{1}{T^{9/4}} \int_T^{T+1} |L_d(r)|^6 \, dr \ll |d|^{1+\varepsilon}
\end{equation}
and
\begin{equation}
	\sum_{\frac{a}{8x} \leq T < \frac{a}{x}} \left(\int_T^{T+1} |L_{d'}(r)|^6 \, dr\right)^{\frac 16}\left(\int_T^{T+1} |L_{d}(r)|^6 \, dr\right)^{\frac 16} \left(\int_T^{T+1} \frac{dr}{r^{3/2-\ep}}\right)^{\frac 23} \ll |dd'|^{\frac 23+\ep}(vw)^{1+\ep}x^{-1+\ep}.
\end{equation}
We also have $K_0^e \ll (vw)^\varepsilon |dd'|^{\frac 16+\varepsilon}$.
These estimates, together with \eqref{eq:Ke1-holder} show that
\begin{equation} \label{eq:K-e-est}
	\mathcal K^e \ll |dd'|^{\frac 23+\ep}(vw)^{1+\ep}x^{-1+\ep}.
\end{equation}

Putting \eqref{eq:K-m-est} and \eqref{eq:K-e-est} together, we find that
\begin{equation}
	\sum_{4\mid c>0} \frac{S_k^+(m,n,c)}{c} \varphi\pfrac{4\pi\sqrt{|mn|}}{c}  \ll \left(|dd'|^{\frac 23}(vw)^{1+\theta}x^{-1} + x^{\frac 12}T^{-\frac 12} + |dd'|^{\frac 5{24}}(vw)^{\frac 14+\frac 12\theta}\right) (|mn|x)^{\ep}.
\end{equation}
To unsmooth the sum of Kloosterman sums, we argue as in \cite{sarnak-tsimerman,aa-kloosterman} to obtain
\begin{equation}
	\sum_{4\mid c>0} \frac{S^+_k(m,n,c)}{c} \varphi\pfrac{4\pi\sqrt{|mn|}}{c} - \sum_{\substack{x\leq c< 2x \\ 4\mid c}} \frac{S^+_k(m,n,c)}{c} \ll \frac{T\log x}{\sqrt x} |mn|^\varepsilon.
\end{equation}
Choosing $T=x^{\frac 23}$ we obtain
\begin{equation} \label{eq:kloo-dyadic-est}
	\sum_{x\leq c< 2x} \frac{S^+_k(m,n,c)}{c} \ll \left(x^{\frac 16} + |dd'|^{\frac 23}(vw)^{1+\theta}x^{-1} + |dd'|^{\frac 5{24}}(vw)^{\frac 14+\frac 12\theta}\right) (|mn|x)^{\ep}.
\end{equation}
To prove Theorem~\ref{thm:Sk-sums} we sum the inital segment $c\leq |dd'|^\alpha(vw)^\beta$ and apply the Weil bound \eqref{eq:weil-bound}, then sum the dyadic pieces for $c\geq |dd'|^\alpha(vw)^\beta$ using \eqref{eq:kloo-dyadic-est}.
To balance the resulting terms we take $\alpha=\frac 49$ and $\beta=\frac 23(1+\theta)$, which gives the bound
\begin{equation} \label{eq:S+-sum-final}
	\sum_{c\leq x} \frac{S^+_k(m,n,c)}{c} \ll \left(x^{\frac 16} + |dd'|^{\frac 29}(vw)^{\frac 13(1+\theta)}\right)(|mn|x)^\varepsilon.
\end{equation}
This completes the proof. \qed

\begin{remark}
	If we assume the Lindel\"of hypothesis for the central values $L(\tfrac 12,v_j\times \chi_d)$, where $v_j$ is the Shimura lift of $u_j$, we can replace the exponent $\tfrac16$ by $0$ in Theorem~\ref{thm:maass-young-est}.
	If, additionally, we assume the Lindel\"of hypothesis for $L(\tfrac 12,\chi_d)$, we can replace the exponent $\frac 29$ by $\frac 16$ in \eqref{eq:S+-sum-final}.
\end{remark}

\appendix
\section{Proof of Theorem~\ref{thm:phicheck}} \label{sec:appendix}

In this section we prove the first estimate of Theorem~\ref{thm:phicheck} and correct some mistakes in the proof given in \cite[Section~6]{aa-kloosterman} for the other three estimates of that theorem.
The mistakes in \cite{aa-kloosterman} all stem from a misunderstanding of the choice of branch cut implicit in Balogh's \cite{balogh} uniform asymptotic expansion for the $K$-Bessel function $K_{iv}(vz)$ when $z\in (0,1)$.
The paper \cite{GST} provides a clear reference for this asymptotic expansion in both regions $z<1$ and $z>1$.
Here we prove corrected statements of some of the propositions in Section~6 of \cite{aa-kloosterman}; the remaining steps of the proof given in that paper follow with only minor changes and are omitted here.

Balogh \cite{balogh} (see also \cite[Section~4]{GST}) gives a uniform asymptotic expansion for $K_{iv}(vz)$ in terms of the Airy function $\Ai$ and its derivative $\Ai'$. 
For $z\in (0,1)$ define
\[
	w(z) = \arccosh\pmfrac 1z - \sqrt{1-z^2} \quad \text{ and } \quad \zeta = \left[\mfrac 32 w(z)\right]^{\frac 23}>0.
\]
Taking $m=1$ in equation (2) of \cite{balogh} (see also \cite[(12)]{GST}) we have
\begin{multline} \label{eq:balogh}
	e^{\frac{\pi v}{2}} K_{iv}(vz) 
	= \frac{\pi\sqrt 2}{v^{\frac 13}} \pfrac{\zeta}{1-z^2}^{\frac 14} 
	\left\{ \Ai\left(-v^{\frac 23}\zeta\right)\left[1+\frac{A(z)}{v^2}\right] \right. \\ + 
	\left.\pmfrac 23^{\frac 13}
	\Ai'\left(-v^{\frac 23}\zeta\right)\frac{B(z)}{v(vw(z))^{\frac 13}}\right\} + 
	O\left(\frac{v^{-\frac 72}}{(1-z^2)^{\frac 14}}\right),
\end{multline}
uniformly for $v\in [1,\infty)$, where
\begin{align}
	A(z) &= \frac{455}{10368\,w(z)^2} - \frac{7(3z^2+2)}{1728(1-z^2)^{\frac 32}w(z)} - \frac{81z^4+300z^2+4}{1152(1-z^2)^3}, \\
	B(z) &=  \frac{3z^2+2}{24(1-z^2)^{\frac 32}} - \frac{5}{72w(z)}.\label{B0def}
\end{align}
Both $A(z)$ and $B(z)$ are $O(1)$ for $z\in (0,1)$.

\begin{proposition} \label{prop:kbessel-asy-small}
Suppose that $z\in (0, 3/4]$ and that $v\geq 1$. Then
\begin{equation} \label{eq:kbessel-asy-small}
	e^{\frac{\pi v}{2}} K_{iv}(vz) =  \frac{\sqrt{2\pi}}{v^{\frac 12}(1-z^2)^{\frac 14}} \left[\cos\left(vw(z)-\mfrac{\pi}{4}\right)+\sin\left(vw(z) - \mfrac{\pi}{4}\right)\frac{3z^2+2}{24v(1-z^2)^{\frac 32}} \right] + O\big( v^{-\frac 52} \big).
\end{equation}
\end{proposition}

\begin{proof}
	Let $c(v,z) = \cos(vw(z) - \pi/4)$ and $s(v,z) = \sin(vw(z)-\pi/4)$.
	By the asymptotic expansions \cite[(9.7.9) and \S 9.7(iii)]{nist} we have
	\begin{gather}
		\Ai(-v^{\frac 23}\zeta) = \pi^{-\frac 12}\left(\mfrac 32vw(z)\right)^{-\frac 16} \left\{c(v,z) + \frac{5s(v,z)}{72vw(z)} + O\left(v^{-2}\right)\right\}, \label{eq:Ai-asymp} \\
		\pmfrac 23^{\frac 13}\frac{\Ai'(-v^{\frac 23}\zeta)}{(vw(z))^{\frac 13}} = \pi^{-\frac 12}\left(\mfrac 32vw(z)\right)^{-\frac 16} \left\{s(v,z) + O\left(v^{-1}\right)\right\}. \label{eq:Ai'-asymp}
	\end{gather}
	Thus by \eqref{eq:balogh} we have
	\begin{equation}
		e^{\frac{\pi v}{2}} K_{iv}(vz) 
		= \\ \frac{\sqrt {2\pi}}{v^{\frac 12} (1-z^2)^{\frac 14}}
		\left\{ 
		c(v,z) + \frac {s(v,z)}v \left(\frac{5}{72w(z)}
		 + 
		B(z)\right)\right\} + 
		O\left(v^{-\frac 52}\right).
	\end{equation}
	Then \eqref{B0def} shows that
	\begin{equation}
		\frac{5}{72w(z)} + B(z) = \frac{3z^2+2}{24(1-z^2)^{\frac 32}},
	\end{equation}
	as desired.
\end{proof}

We require some notation for the next proposition.
Let $J_\nu(x)$ and $Y_\nu(x)$ denote the $J$ and $Y$-Bessel functions, and define 
\[
	M_{\nu}(x) = \sqrt{J_\nu^2(x)+Y_\nu^2(x)}.
\]
\begin{proposition} \label{prop:kbessel-asy-transition}
Suppose that $c>0$.  
Suppose that $v\geq 1$ and that $\frac 3{16}\leq z\leq 1-cv^{-\frac 23}$. Then
\begin{equation} \label{eq:kbessel-asy-transition}
	e^{\frac{\pi v}{2}} K_{iv}(vz) = 
	\pi \frac{w(z)^{\frac 12}}{(1-z^2)^{\frac 14}} M_{\frac 13}(vw(z)) \sin\left(\theta_{\frac 13}(vw(z))\right) + O_c(v^{-4/3}),
\end{equation}
where
$\theta_{\frac 13}(x)$ is a real-valued continuous function satisfying
\begin{equation} \label{eq:theta-1/3-prime}
	\theta_{\frac 13}'(x) = \frac{2}{\pi x M_{\frac 13}^2(x)}.
\end{equation}
\end{proposition}

\begin{proof}
	By \cite[(9.8.1) and (9.8.9)]{nist} we have
	\begin{equation}
		\Ai(-v^{\frac 23}\zeta) 
		= \frac 1{\sqrt{3}} v^{\frac 13}\zeta^{\frac 12} M_{\frac 13}^2(vw(z)) \sin\left[\theta\left(-\left(\mfrac 32vw(z)\right)^{\frac 23}\right)\right],
	\end{equation}
	where $\theta(x)$ is given in \cite[(9.8.11)]{nist}.
	Letting $\theta_{1/3}(x) = \theta\left(-(\frac 32x)^{2/3}\right)$ we find from \cite[(9.8.14)]{nist} that
	\begin{equation}
		\theta_{\frac 13}'(x) = - \left(\mfrac 32x\right)^{-\frac 13} \theta'\left(-\left(\mfrac 32vw(z)\right)^{\frac 23}\right) = \frac{2}{\pi M_{\frac 13}^2(x)}.
	\end{equation}
	For $z \leq 1-cv^{-2/3}$ we have $vw(z)\gg_c 1$.
	Thus it follows from \eqref{eq:Ai-asymp} and \eqref{eq:Ai'-asymp} that after expanding \eqref{eq:balogh}, the terms containing $A(z)$ and $B(z)$ are both $\ll v^{-4/3}$.
	The proposition follows.
\end{proof}

It remains to prove the first estimate of Theorem~\ref{thm:phicheck}.

\begin{proposition}
	Let $\varphi$ be as in Theorem~\ref{thm:phicheck}.
	If $r \leq 1$ then 
	\begin{equation}
		\check\varphi(r) \ll r^{-\frac 32}.
	\end{equation}
\end{proposition}
\begin{proof}
	We begin by splitting the integral as
	\begin{align}
		\check\varphi(r) 
		&= \check\varphi_1(r) + \check\varphi_2(r),
	\end{align}
	where
	\begin{equation}
		\check\varphi_j(r) = \cosh\pi r \int_{I_j} K_{2ir}(u) \varphi(u) \, \mfrac{du}{u},
	\end{equation}
	with $I_1 = [\frac{a}{2(x+T)},\frac{a}{x-T}] \cap [0,r]$ and $I_2 = [\frac{a}{2(x+T)},\frac{a}{x-T}] \cap (r,\infty)$.
	The second piece $\check\varphi_2(r)$ can be estimated exactly as in the proof of Proposition~6.5 of \cite{aa-kloosterman} since $u\geq r$.
	For the first piece we use \cite[(10.27.4) and (10.25.2)]{nist} to get
	\begin{equation}
		|K_{2ir}(u)| \leq \frac{\pi \cosh(2\pi r)^{1/2}}{\sinh(2\pi r)} I_0(u).
	\end{equation}
	From this and the assumption that $r\leq 1$ we obtain
	\begin{equation}
		\check\varphi_1(r) \ll \frac{1}{r} \int_{\frac{a}{2(x+T)}}^{\frac{a}{x-T}} \frac{du}{u} \ll \frac 1r,
	\end{equation}
	where in the last inequality we used that $T\leq x/3$.
	This completes the proof.
\end{proof}


\bibliographystyle{amsplain}

\begin{thebibliography}{10}

\bibitem{aa-kloosterman}
Scott Ahlgren and Nickolas Andersen, \emph{Kloosterman sums and {M}aass cusp
  forms of half integral weight for the modular group}, Int. Math. Res. Not.
  IMRN (2018), no.~2, 492--570. 

\bibitem{andersen-duke-kloosterman}
Nickolas Andersen and William~D. Duke, \emph{Modular invariants for real
  quadratic fields and {K}loosterman sums}, Algebra Number Theory \textbf{14}
  (2020), no.~6, 1537--1575. 

\bibitem{balogh}
Charles~B. Balogh, \emph{Uniform asymptotic expansions of the modified {B}essel
  function of the third kind of large imaginary order}, Bull. Amer. Math. Soc.
  \textbf{72} (1966), 40--43.

\bibitem{baruch-mao}
Ehud~Moshe Baruch and Zhengyu Mao, \emph{A generalized {K}ohnen-{Z}agier
  formula for {M}aass forms}, J. Lond. Math. Soc. (2) \textbf{82} (2010),
  no.~1, 1--16.

\bibitem{blomer}
Valentin Blomer, \emph{Sums of {H}ecke eigenvalues over values of quadratic
  polynomials}, Int. Math. Res. Not. IMRN (2008), no.~16, Art. ID rnn059. 29.
  

\bibitem{bjo}
Jan~Hendrik Bruinier, Paul Jenkins, and Ken Ono, \emph{Hilbert class
  polynomials and traces of singular moduli}, Math. Ann. \textbf{334} (2006),
  no.~2, 373--393. 

\bibitem{nist}
\emph{{NIST Digital Library of Mathematical Functions}}, http://dlmf.nist.gov/,
  Release 1.0.10 of 2015-08-07, Online companion to \cite{nist-book}.

\bibitem{duke-half-integral}
W.~Duke, \emph{Hyperbolic distribution problems and half-integral weight
  {M}aass forms}, Invent. Math. \textbf{92} (1988), no.~1, 73--90.

\bibitem{duke-24}
\bysame, \emph{Modular functions and the uniform distribution of {CM} points},
  Math. Ann. \textbf{334} (2006), no.~2, 241--252. 

\bibitem{dfi-weyl}
W.~Duke, J.~B. Friedlander, and H.~Iwaniec, \emph{Weyl sums for quadratic
  roots}, Int. Math. Res. Not. IMRN (2012), no.~11, 2493--2549. 

\bibitem{DIT-cycle}
W.~Duke, {\"O}.~Imamo{\=g}lu, and {\'A}.~T{\'o}th, \emph{Cycle integrals of the
  {$j$}-function and mock modular forms}, Ann. of Math. (2) \textbf{173}
  (2011), no.~2, 947--981. 

\bibitem{moonshine}
John F.~R. Duncan, Michael~J. Griffin, and Ken Ono, \emph{Moonshine}, Res.
  Math. Sci. \textbf{2} (2015), Art. 11, 57. 

\bibitem{GST}
Amparo Gil, Javier Segura, and Nico~M. Temme, \emph{Computation of the modified
  {B}essel function of the third kind of imaginary orders: uniform {A}iry-type
  asymptotic expansion}, Proceedings of the {S}ixth {I}nternational {S}ymposium
  on {O}rthogonal {P}olynomials, {S}pecial {F}unctions and their {A}pplications
  ({R}ome, 2001), vol. 153, 2003, pp.~225--234. 

\bibitem{kaneko-1}
Masanobu Kaneko, \emph{The {F}ourier coefficients and the singular moduli of
  the elliptic modular function {$j(\tau)$}}, no. 965, 1996, Automorphic forms
  on algebraic groups (Japanese) (Kyoto, 1995), pp.~172--177. 

\bibitem{kaneko-2}
\bysame, \emph{Traces of singular moduli and the {F}ourier coefficients of the
  elliptic modular function {$j(\tau)$}}, Number theory ({O}ttawa, {ON}, 1996),
  CRM Proc. Lecture Notes, vol.~19, Amer. Math. Soc., Providence, RI, 1999,
  pp.~173--176. 

\bibitem{kohnen-newforms}
Winfried Kohnen, \emph{Newforms of half-integral weight}, J. Reine Angew. Math.
  \textbf{333} (1982), 32--72. 

\bibitem{kuznetsov}
N.~V. Kuznetsov, \emph{The {P}etersson conjecture for cusp forms of weight zero
  and the {L}innik conjecture. {S}ums of {K}loosterman sums}, Mat. Sb. (N.S.)
  \textbf{111(153)} (1980), 334--383. 

\bibitem{nist-book}
F.~W.~J. Olver, D.~W. Lozier, R.~F. Boisvert, and C.~W. Clark (eds.),
  \emph{{NIST Handbook of Mathematical Functions}}, Cambridge University Press,
  New York, NY, 2010, Print companion to \cite{nist}.

\bibitem{petersson}
Hans Petersson, \emph{\"{U}ber die {E}ntwicklungskoeffizienten der automorphen
  {F}ormen}, Acta Math. \textbf{58} (1932), no.~1, 169--215. 

\bibitem{proskurin-new}
N.~V. Proskurin, \emph{On general {K}loosterman sums}, Zap. Nauchn. Sem.
  S.-Peterburg. Otdel. Mat. Inst. Steklov. (POMI) \textbf{302} (2003), no.~19,
  107--134.

\bibitem{rademacher-j}
Hans Rademacher, \emph{The {F}ourier {C}oefficients of the {M}odular
  {I}nvariant {J}({$\tau$})}, Amer. J. Math. \textbf{60} (1938), no.~2,
  501--512. 

\bibitem{sarnak-tsimerman}
Peter Sarnak and Jacob Tsimerman, \emph{On {L}innik and {S}elberg's conjecture
  about sums of {K}loosterman sums}, Algebra, arithmetic, and geometry: in
  honor of {Y}u. {I}. {M}anin. {V}ol. {II}, Progr. Math., vol. 270,
  Birkh{\"a}user Boston, Inc., Boston, MA, 2009, pp.~619--635.

\bibitem{young-subconvexity}
Matthew~P. Young, \emph{Weyl-type hybrid subconvexity bounds for twisted
  {$L$}-functions and {H}eegner points on shrinking sets}, J. Eur. Math. Soc.
  (JEMS) \textbf{19} (2017), no.~5, 1545--1576. 

\bibitem{zagier-traces}
Don Zagier, \emph{Traces of singular moduli}, Motives, polylogarithms and
  {H}odge theory, {P}art {I} ({I}rvine, {CA}, 1998), Int. Press Lect. Ser.,
  vol.~3, Int. Press, Somerville, MA, 2002, pp.~211--244. 

\end{thebibliography}

\providecommand{\bysame}{\leavevmode\hbox to3em{\hrulefill}\thinspace}
\providecommand{\MR}{\relax\ifhmode\unskip\space\fi MR }

\end{document}